\documentclass[12pt]{amsart}

\usepackage{amsmath,amsthm,amscd,euscript}

\def \r{\mathbb R}
\def \q{\mathbb Q}
\def \z{\mathbb Z}
\def \n{\mathbb N}

\newtheorem{statement}{Statement}[section]
\newtheorem{proposition}{Proposition}[section]

\newtheorem{hyp}{Hypothesis}
\theoremstyle{remark}

\theoremstyle{plane}
\newtheorem{example}{Example}[section]

\title
[Two-dimensional periodic continued fractions]
{On tori triangulations associated with two-dimensional continued fractions
of cubic irrationalities.}
\author{O.~N.~Karpenkov}
\thanks{Supported by INTAS-00-0259, SS-1972.2003.1 and RFBR-01-01-00803
projects}
\email[Oleg Karpenkov]{karpenk@mccme.ru}

\begin{document}
\input epsf
\maketitle

\section*{Introduction.}

A series of properties for ordinary continued fractions possesses
multidimensional analogues.
H.~Tsuchihashi~\cite{Tsu} showed the connection between
periodic multidimensional continued fractions and multidimensional cusp
singularities.
The relation between sails of multidimensional continued fractions
and Hilbert bases is describe by J.-O.~Moussafir in the work~\cite{Mus}.

In his book~\cite{Arn1} dealing with theory of continued fractions
V.~I.~Arnold gives various images for the sails of two-dimensional
continued fraction generalizes the golden ratio.
In the article~\cite{Kor2} E.~I.~Korkina investigated the sales
for the simplest two-dimensional continued fractions of cubic irrationalities,
whose fundamental region consists of two triangles, three edges
and one vertex.

We consider the same model of the multidimensional continued fraction
as was considered by the authors mentioned above.
In the present work the examples of new tori triangulation of the
sails for two-dimensional continued fractions of cubic irrationalities
for some special families possessing the fundamental regions
with more complicated structure are obtained.

In $\S$1 the necessary definitions and notions are given.
In $\S$2 the properties of two-dimensional continued fractions
constructed using Frobenius operators are investigated,
the relation between the equivalence classes of
tori triangulations  and cubic extensions for the field of
rational numbers is discussed.
(The detailed analysis of the properties of cubic extensions
for the rational numbers field and their classification
is realized by B.~N.~Delone and D.~K.~Faddeev in the work~\cite{Del1}.)
In $\S$3 the appearing examples of tori triangulations are discussed.

The author is grateful to professor V.~I.~Arnold for constant attention
for this work and for useful remarks.

\section{Definitions.}

Points of the space $\r^k$  ($k\ge 1$) whose coordinates are all integer
numbers are called {\it integer points}.

Consider a set of $n+1$ hyperplanes passing through the origin
in general position in the space $\r^{n+1}$.
The complement to these hyperplanes consists of
$2^{n+1}$ open orthants. Let us choose an arbitrary orthant.

The boundary of the convex hull of all integer points except
the origin in the closure of the orthant is called {\it the sail}.

The union of all $2^{n+1}$ sails defined by these hyperplanes
of the space $r^{n+1}$ is called
{\it $n$-dimensional continued fraction} constructed according to
the given $n+1$ hyperplanes in general position in
$n+1$-dimensional space.

Two $n$-dimensional sails (continued fractions) are
called {\it equivalent} if there exists
a linear integer lattice preserving transformation
of the $n+1$-dimensional space such that it maps one
sail (continued fraction) to the other.

To construct the whole continued fraction up to the equivalence relation
in one-dimensional case it is sufficiently  to know
some integer characteristics of one sail
(that is to say the integer lengths of the edges and the integer
angles between the consecutive edges of one sail).

\begin{hyp}\label{hyp1}{\bf (Arnold)}
There exist the collection of integer characteristics of the sail
that uniquely up to the equivalence relation determines the continued fraction.
\end{hyp}

Let $A \in GL(n+1,\r)$ be an operator  whose roots are all real and distinct.
Let us take the $n$-dimensional spaces that spans all possible subsets
of $n$ linearly independent eigenvectors of the operator $A$.
As far as eigenvectors are linearly independent,
the obtained $n+1$ hyperspaces will be $n+1$ hyperspaces
in general position.
The multidimensional continued fraction is constructed just with
respect to these hyperspaces.

\begin{proposition}
Continued fractions constructed by operators $A$ and $B$ of the
group $GL(n+1,\r)$ with distinct real irrational eigenvalues
are equivalent iff there exists such an integer operator $X$ with
the determinant equals to one that the operator $\tilde A$ obtained
from the operator $A$ by means of the conjugation by the operator $X$
commutes with $B$.
\end{proposition}

\begin{proof}
Let the continued fractions constructed by operators $A$ and $B$ of the
group $GL(n+1,\r)$ with distinct real irrational eigenvalues
are equivalent, i. e. there exists linear integer lattice preserving
transformation of the space that maps the continued fraction of the operator
$A$ to the continued fraction of the operator $B$
(and the orthants of the first continued fraction maps to the orthants of
the second one).
Under such transformation of the space
the operator $A$ conjugates by some integer operator $X$ with the determinant
equals to one.
All eigenvalues of the obtained operator $A$ are distinct and real
(since the characteristic polynomial of the operator is preserving).
As far as the orthants of the first continued fraction maps to the
orthants of the second one, the sets of the proper directions
for the operators $\tilde A$ and $B$ coincides. Thus, given operators are
diagonalizable together in some basis and hence they commutes.

Let us prove the converse. Suppose there exists such
an integer operator $X$ with the determinant equals to one
that the operator $\tilde A$ obtained from the operator $A$ by means of
the conjugation by the operator $X$ commutes with $B$.
Note that the eigenvalues of the operators $A$ and $\tilde A$ coincide.
Therefore all eigenvalues of the operator $\tilde A$
(just as for the operator $B$) are real, distinct, and irrational.
Let us consider such basis that the operator $\tilde A$ is diagonal in it.
Simple verification shows that the operator $B$ is also diagonal in this
basis.
Consequently the operators $\tilde A$ and $B$ defines the same orthant
decomposition of the $n+1$-dimensional space and
the operators corresponding to this continued fractions coincide.
It remains just to note that a conjugation by an integer operator with the
determinant equals to one corresponds to the linear
integer lattice preserving transformation of the $n+1$-dimensional space.
\end{proof}

Later on we will consider only continued fractions constructed by
some invertible integer operator of the $n+1$-dimensional space
such that its inverse is also integer.
The set of such operators form the group denoted by $GL(n+1,\z)$.
This group consist of the integer operators with the determinants equal
to $\pm 1$.

The $n$-dimensional continued fraction constructed by
the operator $A \in GL(n+1,\z)$ whose characteristic polynomial
over the rational numbers field
is irreducible and eigenvalues are real is called
{\it the $n$-dimensional continued fraction of $(n+1)$-algebraic
irrationality}.
The cases of $n=1,2$ correspond to {\it one$($two$)$-dimensional
continued fractions of quadratic $($cubic$)$ irrationalities}.

Let the characteristic polynomial of the operator $A$ be
irreducible over the rational numbers field
and its roots be real and distinct.
Under the action of the integer operators
commuting with $A$ whose determinants
are equal to one and preserving the given sail
the sail maps to itself.
These operators form an Abelian group.
It follows from Dirichlet unity elements theorem (см.~\cite{BSh})
that this group is isomorphic to $\z^n$ and that its action is free.
The factor of a sail under such group action is isomorphic to
$n$-dimensional torus at that.
(For the converse see~\cite{Kor1} and~\cite{Tsu}.)
The polyhedron decomposition of $n$-dimensional torus is defined
in the natural way, the affine types of the polyhedrons are also defined
(in the notion of the affine type we include the number and mutual
arrangement of the integer points for the faces of the polyhedron).
In the case of two-dimensional continued fractions for cubic irrationalities
such decomposition is usually called torus {\it triangulation}.

By {\it a fundamental region} of the sail we call a union of some
faces that contains exactly one face from each equivalence class.

\section{ Conjugacy classes of two-dimensional continued fractions
for cubic irrationalities.}

Two-dimensional continued fractions for cubic irrationalities
constructed by the operators $A$ and $-A$ coincide.
In that way the study of continued fractions for integer operators
with the determinants equal to $\pm 1$ reduces to the study of
continued fractions for integer operators
with the determinants equal to one (i. e. operators of the group $SL(3,\z)$).

An operator (matrix) with the determinant equals to one
$$
A_{m,n}:=
\left(
\begin{array}{ccc}
0 &1 &0 \\
0 &0 &1 \\
1 &-m &-n \\
\end{array}
\right),
$$
where $m$ and $n$ are arbitrary integer numbers is called
{\it a Frobenius operator $($matrix$)$}.
Note the following: if the characteristic polynomial $\chi_{A_{m,n}(x)}$
is irreducible over the field $\q$ than
the matrix for the left multiplication by the element $x$
operator in the natural basis $\{1,x,x^2 \}$
in the field $\q[x]\big/ \big( \chi_{A_{m,n}(x)}\big)$
coincides with the matrix $A_{m,n}$.

Let an operator $A\in SL(3,\z)$ has distinct real irrational
eigenvectors.
Let $e_1$ be some integer nonzero vector, $e_2=A(e_1)$, $e_3=A^2(e_1)$.
Then the matrix of the operator in the basis $(e_1, e_2, c e_3)$
for some rational $c$ will be Frobenius. However the transition
matrix here could be non-integer and the corresponding continued fraction
is not equivalent to initial one.

\begin{example}
The continued fraction constructed by the operator
$$
A=
\left(
\begin{array}{ccc}
1  &2 &0  \\
0  &1 &2  \\
-7 &0 &29 \\
\end{array}
\right),
$$
is not equivalent to the continued fraction constructed by any
Frobenius operator with the determinant equals to one.
\end{example}

Thereupon the following question is of interest.
{\it How often the continued fractions that don't correspond to
Frobenius operators can occur?}

In any case the family of Frobenius operators possesses some
useful properties that allows us to construct
whole families of nonequivalent two-dimensional periodic
continued fractions at once, that extremely actual itself.

It is easy to obtain the following statements.

\begin{statement}
The set $\Omega$ of operators $A_{m,n}$ having
all eigenvalues real and distinct is defined by the inequality
$n^2m^2-4m^3+4n^3-18mn-27\le 0$.
For the eigenvalues of the operators of the set to be irrational
it is necessary to subtract extra two perpendicular lines in the
integer plain: $A_{a,-a}$ and $A_{a,a+2}$, $a \in \z$.
\end{statement}

\begin{statement}\label{p1}
The two-dimensional continued fractions for the cubic irrationalities
constructed by the operators $A_{m,n}$ and $A_{-n,-m}$ are equivalent.
\end{statement}

Further we will consider all statements modulo this symmetry.

{\it Remark.} Example~\ref{primer21} given below shows that
among periodic continued fractions constructed by operators in the set
$\Omega$ equivalent continued fractions can happen.

Let us note that there exist nonequivalent two-dimensional periodic
continued fractions constructed by operators of the group $GL(n+1,\r)$
whose characteristic polynomials define isomorphic extensions of the
rational numbers field.
In the following example the operators with equal characteristic polynomials
but distinct continued fractions are shown.

\begin{example}\label{primer}
The operators $(A_{-1,2})^3$ and $A_{-4,11}$ have distinct two-dimensional
continued fractions $($although their characteristic polynomials coincides$)$.
\end{example}

At the other hand similar periodic continued fractions can correspond to
operators with distinct characteristic polynomials.

\begin{example}\label{primer21}
The operators $A^2_{0,-a}$ and $A_{-2a,-a^2}$
are conjugated by the operator in the group $GL(3,\z)$
and hence the periodic continued fractions
$($including the torus triangulations$)$ corresponding to the operators
$A_{0,-a}$ and $A_{-2a,-a^2}$ are equivalent.
\end{example}

Let us note that distinct cubic extensions of the field $\q$
posses nonequivalent triangulations.

\section{Torus triangulations and fundamental regions
for some series of operators $A_{m,n}$}

Torus triangulations and fundamental domains for several infinite series
of Frobenius operators are calculated here.
In this paragraph it considers only the sails containing the point $(0,0,1)$
in its convex hull.

The ratio of the Euclidean volume for
an integer $k$-dimensional polyhedra in $n$-dimensional space
to the Euclidean volume for the minimal simplex in the same
$k$-dimensional subspace is called its {\it the integer $k$-dimensional
volume}
(if $k=1$ --- {\it the integer length} of the segment,
if $k=2$ --- {\it the integer area} of the polygon).
\\
The ratio of the Euclidean distance from an integer hyperplane
(containing an $n-1$-dimensional integer sublattice) to the integer point
to the minimal Euclidean distance from the hyperplane to some
integer point in the complement of this hyperplane
is called the corresponding {\it integer distance}.
\\
By {\it the integer angle} between two integer rays (i.e. rays that contain
more than one integer point) with the vertex at the same integer point
we call the value
$S(u,v)/(|u|\cdot|v|)$, where $u$ and $v$ are arbitrary
integer vectors passing along the rays
and $S(u,v)$ is the integer volume of the triangle with edges $u$ and $v$.

{\it Remark.}
Our integer volume is an integer number
(in standard parallelepiped measuring the value will be $k!$ times less).
The integer $k$-dimensional volume of the simplex is equal to the
index of the lattice subgroup generated by its edges having
the common vertex.

Since the integer angles of any triangle with all integer vertices
can be uniquely restored by the integer lengths of the triangle
and its integer volume we will not write the integer
angles of triangles below.

\begin{hyp}
The specified invariants distinguish all nonequivalent torus triangulations
of two-dimensional continued fractions for the cubic irrationalities.
\end{hyp}

In the formulations of the propositions~\ref{t31}---\ref{t35}
we say only about homeomorphic type for the torus triangulations
although the description of the fundamental regions allowing us
to calculate any other invariant including affine types of the faces
is given in the proofs.
(As an example we calculate integer volumes and distances
to faces in propositions~\ref{t31} and~\ref{t32}.)
The affine structure examples of all triangulation
faces are shown on the figures.

\begin{proposition}\label{t31}
Let
$m=b-a-1$, $n=(a+2)(b+1)$ $(a,b\ge 0)$
then the torus triangulation corresponding to the operator $A_{m,n}$
is homeomorphic to the following one:
$$
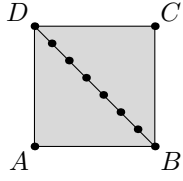
\epsfbox{fig2.3}$$
$($in the figure $b=6)$.
\end{proposition}

\begin{proof}
The operators
$$
X_{a,b}=A_{m,n}^{-2},
\quad
Y_{a,b}=A_{m,n}^{-1}\big(A_{m,n}^{-1}-(b+1)I\big)
$$
commutes with the operator $A_{m,n}$ without
transpose of the sails
(note that the operator $A_{m,n}$ transpose the sails).
Here $I$ is the identity element of the group $SL(3,\z)$.

Let us describe the closure for one of the fundamental regions obtaining
by factoring of the sail by the operators
$X_{a,b}$ and $Y_{a,b}$.
Consider the points $A=(1,0,a+2)$, $B=(0,0,1)$, $C=(b-a-1,1,0)$
and $D=((b+1)^2,b+1,1)$ of the sail containing the point $(0,0,1)$.
Under the operator $X_{a,b}$ action
the segment $AB$ maps to the segment $DC$ (the point $A$ maps to
the point $D$ and $B$ to $C$).
Under the operator $Y_{a,b}$ action
the segment $AD$ maps to the segment $BC$ (the point $A$ maps to
the point $B$ and $D$ to $C$).
The integer points $((b+1)i,i,1)$,
where $i\in \{ 1,\ldots ,b\}$ belong to the interval $BD$.

As can be easily seen the integer lengths of the segments $AB$,
$BC$, $CD$, $DA$ and $BD$
are equal to 1, 1, 1, 1 and $b+1$ correspondingly;
the integer areas of both triangles $ABD$ and $BCD$ are equal to $b+1$.
The integer distances from the origin to the plains containing
the triangles $ABD$ and $BCD$ are equal to $1$ and $a+2$ correspondingly.

The operators $X_{a,b}$ and $Y_{a,b}$ mapping the sail to itself,
since all their eigenvectors are positive
(or in this case it is equivalent to say that values
of their characteristic polynomials on negative semi-axis are always
negative).
Furthermore this operators are the generators
of the group of integer operators mapping the sail to itself, since
it turn out that torus triangulation obtaining
by factoring the sail by this operators contains the unique vertex
(zero-dimensional face),
and hence the torus triangulation have no smaller subperiod.
\end{proof}

Let us show that all vertices for the fundamental domain of the
arbitrary periodic continued fraction can be chosen from the closed
convex hull of the following points: the origin; $A$; $X(A)$;
$Y(A)$ and $XY(A)$, where $A$ is the arbitrary zero-dimensional
face of the sail, and operators $X$ and $Y$ are generators
of the group of integer operators mapping the sail to itself.

Consider a tetrahedral angle with the vertex at the origin and
edges passing through the points $A$, $X(A)$, $Y(A)$, and $XY(A)$.
The union of all images for this angle under the transformations
of the form $X^m Y^n$, where $m$ and $n$ are integer numbers,
covers the whole interior of the orthant.
Hence all vertices of the sail can be obtained by shifting
by operators $X^mY^n$ the vertices of the sail lying in our tetrahedral angle.
The convex hull for the integer points of the form $X^mY^n(A)$
is in the convex hull of all integer points for the given orthant at that.
Therefore the boundary of the convex hull for all integer points of the orthant
is in the complement to the interior points of the convex hull for the
integer points of the form $X^mY^n(A)$.
The complement in its turn is in the unit of all images for the convex hull
of the following points: the origin, $A$, $X(A)$, $Y(A)$, and $XY(A)$,
under the transformations of the form $X^m Y^n$, where $m$ and $n$ are
integer numbers.

It is obviously that all points of the constructed polyhedron except the origin
lie at the concerned open orthant at that.

\begin{proposition}\label{t32}
Let $m=-a$, $n=2a+3$ $(a\ge 0)$, then the torus triangulation
corresponding to the operator $A_{m,n}$ is homeomorphic to the following one:
$$
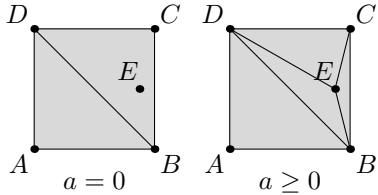
\epsfbox{fig2.4}$$
\end{proposition}

\begin{proof}
Let us choose the following generators of the subgroup of integer operators
mapping the sail to itself:
$$
X_{a}=A_{m,n}^{-2};
\quad
Y_{a}=(2I-A_{m,n}^{-1})^{-1}.
$$

As in the previous case let us make the closure of one of the fundamental
regions of the sail
(containing the point $(0,0,1)$) that obtains by factoring by the operators
$X_{a}$ and $Y_{a}$.
Let $A=(0,0,1)$, $B=(2,1,1)$, $C=(7,4,2)$ and $D=(-a,1,0)$.
Besides this points the vertex $E=(3,2,1)$ is in the fundamental
region.
Under the operator $X_{a}$ action
the segment $AB$ maps to the segment $DC$ (the point $A$ maps to
the point $D$ and $B$ --- to $C$).
Under the operator $Y_{a}$ action
the segment $AD$ maps to the segment $BC$ (the point $A$ maps to
the point $B$ and $D$ --- to $C$).

If $a=0$ then the integer length of the sides $AB$, $BC$, $CD$ and $DA$
are equal to 1, and the integer areas of the triangles $ABD$ and
$BCD$ are equal to 1 and 3 correspondingly.
The integer distances from the origin to the plains containing
the triangles $ABD$ and $BCD$ are equal to $2$ and $1$ correspondingly.

If $a>0$ then all integer length of the sides and integer areas of all
four triangles are equal to 1.
The integer distances from the origin to the plains containing
the triangles $ABD$, $BDE$, $BCE$ and $CED$ are equal
to $a+2$, $a+1$, $1$ and $1$ correspondingly.

Here and below the proofs of the statements on the generators
are similar to the proof of the corresponding statements in the proof
of proposition~\ref{t31}.
\end{proof}

\begin{proposition} \label{t33}
Let $m=2a-5$, $n=7a-5$ $(a\ge 2)$, then the torus triangulation
corresponding to the operator $A_{m,n}$ is homeomorphic
to the following one:
$$
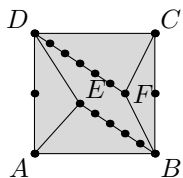
\epsfbox{fig2.5}$$
$($in the figure $a=5)$.
\end{proposition}

\begin{proof}
Let us choose the following generators of the subgroup of integer operators
mapping the sail to itself:
$$
X_{a}=2A_{m,n}^{-1}+7I;
\quad
Y_{a}=A_{m,n}^2.
$$

Let us make the closure of the fundamental
regions of the sail
(containing the point $(0,0,1)$) that obtains by factoring by the operators
$X_{a}$ and $Y_{a}$.
Let $A=(-14,4,-1)$, $B=(-1,1-a,7a^2-10a+4)$, $C=(1,5-7a,49a^2-72a+30)$
and $D=(0,0,1)$.
Under the operator $X_{a}$ action
the segment $AB$ maps to the segment $DC$ (the point $A$ maps to
the point $D$ and the $B$ --- to $C$).
Under the operator $Y_{a}$ action
the segment $AD$ maps to the segment $BC$ (the point $A$ maps to
the point $B$ and $D$ --- to $C$).
Besides this points the vertices $E=(-1,0,2a-1)$ and $F=(0,-a,7a^2-5a+1)$
are in the fundamental region.
The interval $BE$ contains $a-2$ integer points, the interval $DF$
--- $a-1$, $AD$ and $CB$ --- one point for each.
\end{proof}

\begin{proposition} \label{t34}
Let $m=a-1$, $n=3+2a$ $(a\ge 0)$, then the torus triangulation
corresponding to the operator $A_{m,n}$ is homeomorphic to the
following one:
$$
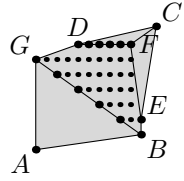
\epsfbox{fig2.6}$$
$($in the figure $a=4)$.
\end{proposition}

\begin{proof}
Let us choose the following generators of the subgroup of integer operators
mapping the sail to itself:

$$
X_{a}=(2I+A_{m,n}^{-1})^{-2};
\quad
Y_{a}=A_{m,n}^{-2}.
$$

Let us make the closure of one of the fundamental
regions of the sail
(containing the point $(0,0,1)$) that obtains by factoring by the operators
$X_{a}$ and $Y_{a}$.
Let $A=(1,-2a-3,4a^2+11a+10)$, $B=(0,0,1)$, $C=(-4a-11,2a+5,-a-2)$
and $D=(-a-2,0,a^2+3a+3)$.
Besides this points the vertices
$E=(-2,1,0)$, $F=(-2a-3,a+1,1)$ and $G=(0,-1-a,2a^2+5a+4)$
are in the fundamental region.
The intervals $BG$ and $DF$ contains $a$ integer points each.
In the interior of the pentagon $BEFDG$ $(a+1)^2$
integer points of the form:
$(-j,-i+j,(2a+3)i-(a+2)j+1)$, where $1\le i\le a+1$, $1\le j\le 2i-1$
are contained.
Under the operator $X_{a}$ action
the segment $AB$ maps to the segment $DC$ (the point $A$ maps to
the point $D$ and $B$ to $C$).
Under the operator $Y_{a}$ action
the broken line $AGD$ maps to the broken line $BEC$ (the point $A$ maps to
the point $B$, the point $G$ maps to the point $E$,
and the point $D$ --- to the point $C$).
\end{proof}

\begin{proposition}\label{t35}
Let $m=-(a+2)(b+2)+3$, $n=(a+2)(b+3)-3$ $(a\ge 0$, $b\ge 0)$,
then the torus triangulation corresponding to the operator $A_{m,n}$
is homeomorphic to the following one:
$$
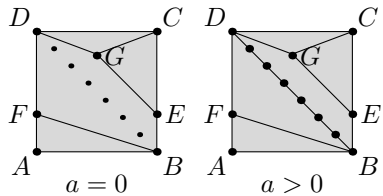
\epsfbox{fig2.7}$$
$($in the figure $b=5)$.
\end{proposition}

\begin{proof}
Let us choose the following generators of the subgroup of integer operators
mapping the sail to itself:
$$
X_{a,b}=((b+3)I-(b+2)A_{m,n}^{-1})A_{m,n}^{-2};
\quad
Y_{a,b}=A_{m,n}^{-2}.
$$

Let us make the closure of one of the fundamental
regions of the sail
(containing the point $(0,0,1)$) that obtains by factoring by the operators
$X_{a,b}$ and $Y_{a,b}$.
Let $A=(b^2+3b+3, b^2+2b-a+1,a^2b+3a^2+4ab+b^2+6a+5b+4)$,
$B=(b^2+5b+6,b^2+4b+4)$, $C=(-ab-2a-2b-1,1,0)$ and $D=(0,0,1)$.
The interval $BD$ contains $b+1$ integer points.
Besides this points the vertices
$E=(b+4,b+3,b+2)$, $F=(b+2,b+1,a+b+2)$ and $G=(1,1,1)$
are in the fundamental region.
Under the operator $X_{a,b}$ action
the segment $AB$ maps to the segment $DC$ (the point $A$ maps to
the point $D$ and the point $B$ --- to the point $C$).
Under the operator $Y_{a,b}$ action
the broken line $AFD$ maps to the broken line $BEC$ (the point $A$ maps to
the point $B$, the point $F$ maps to the point $E$,
and the point $D$ --- to the point $C$).
\end{proof}

Note that the generators of the subgroup of operators
commuting with the operator $A_{m,n}$ that do not transpose the sails
can be expressed by the operators $A_{m,n}$ and
$\alpha I +\beta A^{-1}_{m,n}$,
where $\alpha$ and $\beta$ are nonzero integer numbers.

It turns out that in general case the following statement is true:
the determinants of the matrices for the operators
$\alpha I +\beta A^{-1}_{m,n}$ and
$\alpha I +\beta A^{-1}_{m+k\beta,n+k\alpha}$ are equal.
In particular, if the absolute value of the determinant of the matrix
for the operator $\alpha I +\beta A^{-1}_{m,n}$ is equal to one
then the absolute value of the determinant of the matrix
for the operator
$\alpha I +\beta A^{-1}_{m+k\beta,n+k\alpha}$ is also equal to one
for an arbitrary integer $k$.

Seemingly torus triangulation for the other sequences of operators \\
$A_{m_0+\beta s,n_0+\alpha s}$, where $s \in \n$,
(besides considered in the propositions~\ref{t31}---\ref{t35})
have much in common (for example, number of polygons and their types).

Note that the numbers $\alpha$ and $\beta$ for such sequences
satisfy the following interesting property. Since
$$
|\alpha I +\beta A_{m,n}^{-1}|=
\alpha^3+\alpha^2\beta m-\alpha \beta^2n+\beta^3,
$$
the integer numbers $m$ and $n$ such that
$|\alpha^3+\alpha^2\beta m-\alpha \beta^2n+\beta^3|=1$
exist iff $\alpha^3-1$ is divisible by $\beta$
and $\beta^3-1$ is divisible by  $\alpha$,
or $\alpha^3+1$ is divisible by $\beta$ and
$\beta^3+1$ is divisible by $\alpha$.

For instance such pairs $(\alpha,\beta)$ for $10 \ge \alpha \ge \beta \ge -10$
(besides described in the propositions~\ref{t31}---\ref{t35})
are the following:
$(3,2)$, $(7,-2)$, $(9,-2)$, $(9,2)$, $(7,-4)$, $(9,4)$, $(9,5)$, $(9,7)$.

\begin{figure}
$$\epsfbox{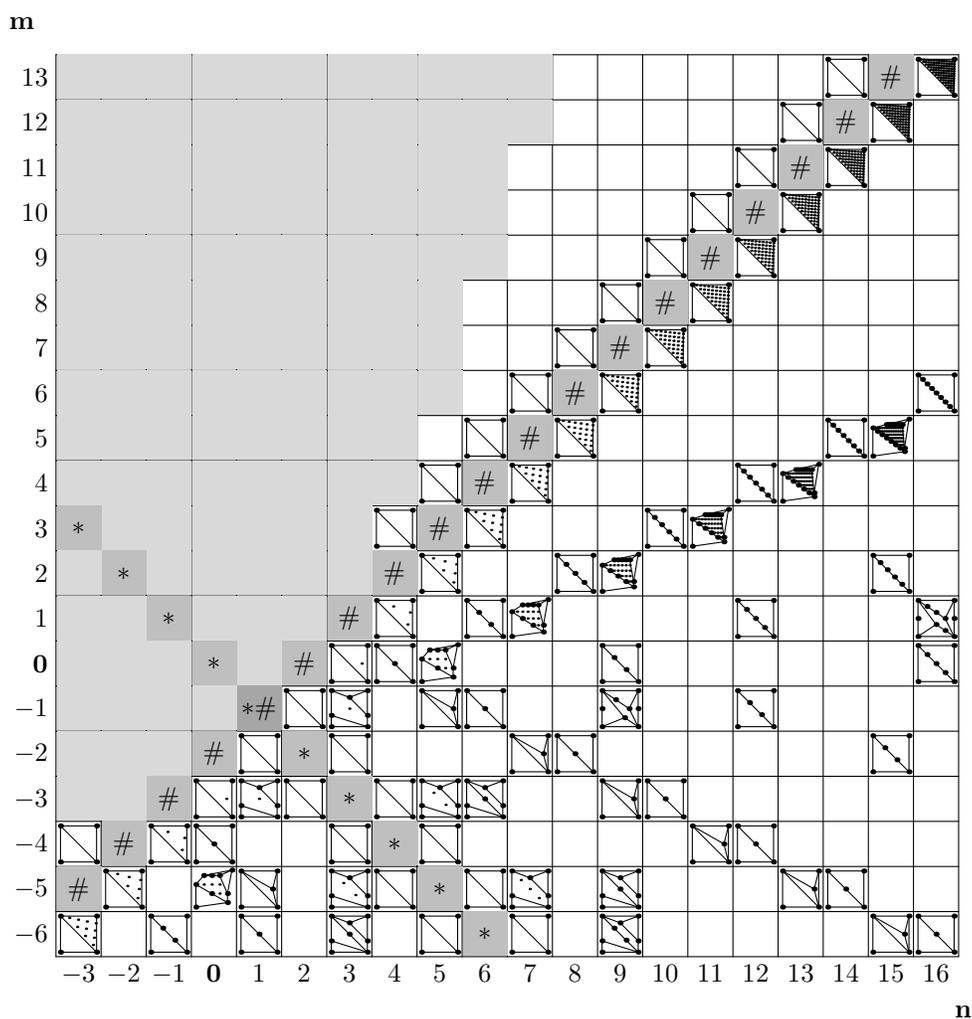}$$
\caption{Torus triangulations for operators $A_{m,n}$.}\label{pic2}
\end{figure}

In conclusion we show the table with squares filled with torus triangulation
of the sails constructed in this work whose convex hulls contain
the point with the coordinates $(0,0,1)$, see fig~\ref{pic2}.
The torus triangulation for the sail of the two-dimensional continued fraction
for the cubic irrationality, constructed by the operator $A_{m,n}$
is shown in the square sited at the intersection of the string with number $n$
and the column with number $m$.
If one of the roots of characteristic polynomial for the operator is equal to
1 or -1 at that than we mark the square $(m,n)$ with the sign $*$
or $\#$ correspondingly.
The squares that correspond to the operators which characteristic
polynomial has two complex conjugate roots
we paint over with light gray color.

\end{document}